\documentclass[]{amsart}
\usepackage{xspace}
\usepackage[all]{xy}
\usepackage{amssymb}
\usepackage{stmaryrd}

\usepackage{enumerate}

\theoremstyle{definition}
\newtheorem{definition}{Definition}[section]
\newtheorem{rem}[definition]{Remark}

\theoremstyle{plain}
\newtheorem{lem}[definition]{Lemma}
\newtheorem{prop}[definition]{Proposition}
\newtheorem{thm}[definition]{Theorem}
\newtheorem{cor}[definition]{Corollary}

\newtheorem{ques}[definition]{Question}

\newtheorem{fact}[definition]{Fact}
\newtheorem{cond}[definition]{Conditions}

\newcommand{\RMod}{R\rm{-Mod}}

\newcommand{\RImp}{R\rm{-Imp}}

\newcommand{\ModR}{{\rm Mod-}R}

\newcommand{\Hom}{{\rm Hom}}
\newcommand{\Tor}{{\rm Tor}_1}
\newcommand{\Ext}{{\rm Ext}^1}

\newcommand{\D}{{\rm{D}}}

\newcommand{\tor}{{\rm{t}}}
\newcommand{\s}{{\Theta}}
\newcommand{\T}{{\rm{T}}}

\newcommand{\Sreg}{S_{reg}}
\newcommand{\Div}{{\cal Div}}
\newcommand{\Tf}{\cal T\!f}
\newcommand{\To}{\cal T}

\newcommand{\ann}{{\rm{ann}}}
\newcommand{\rann}{{\mathfrak{r}}}
\newcommand{\lann}{{\mathfrak{l}}}

\newcommand{\cal}{\mathcal}

\renewcommand{\bar}{\overline}

\renewcommand{\phi}{\varphi}

\renewcommand{\to}{\longrightarrow}

\begin{document}

\title{
Torsion-free, divisible, and Mittag-Leffler modules}
\author{Philipp Rothmaler}
\date{August 9, 2012}
\dedicatory{To Leonell}
\thanks{This work was supported in part by PSC-CUNY Award  64595-00 42.} 

\begin{abstract} 
We study (relative) $\cal K$-Mittag-Leffler modules, with emphasis on the class $\cal K$ of absolutely pure modules. A final goal is to describe the $\cal K$-Mittag-Leffler abelian groups as those that are, modulo their torsion part, $\aleph_1$-free, Cor.\ref{AG}. Several more general results of independent interest are derived on the way. In particular, every flat  $\cal K$-Mittag-Leffler module (for $\cal K$ as before) is Mittag-Leffler, Thm.\ref{alreadyML}. A question about the definable subcategories generated by the divisible modules and the torsion-free modules, resp., has been left open, Quest.\ref{ques}.
\end{abstract}

\maketitle

Suppose $R$ is an arbitrary associative ring with $1$ and $M$ a unital left $R$-module. $M$ is said to be \emph{torsion-free} if, for all $r\in R$ and $m\in M$, the annihilation $rm=0$ is possible only if $m$ is a linear combination $\sum r_i m_i$ of certain $m_i\in M$ with coefficients $r_i\in R$ already annihilated by $r$, i.e., such that $r r_i=0$. This may be stated more succinctly thus: $\ann_M(r)\subseteq \rann(r)M$,
i.e., if $rm=0$, then   $\rann(r)|m$---here  $\rann(r)$ denotes the right annihilator of $r$ in $R$. This is the same as the usual definition when no zero-divisors are around.

This definition goes back at least to Hattori, who proved that it is equivalent to  $\Tor(R/rR, M)=0$ for all $r\in R$ \cite[Prop.1]{Hat}.  

Dually, $M$ is said to be \emph{divisible} if a ring element  $r$ divides an element $m\in M$ whenever the left annihilator $\lann(r)$  of $r$ in $R$ annihilates $m$, i.e., if $\lann(r) m=0$ then $r|m$  \cite[Def.3.16]{L}.

This goes back to the same paper where it is proved to be equivalent to  the condition $\Ext(R/Rr, M)=0$ for all $r\in R$ \cite[Prop.$1^\prime$]{Hat}.

The first aim of the paper is to characterize the rings over which these two classes form \emph{definable subcategories}, i.e., classes closed under direct product, direct limit and pure submodule (considered always as full subcategories of the category of all modules).
We prove that the class $\Div=\Div_R$  of divisible right $R$-modules forms a definable subcategory of $\ModR$ if and only if the class  $\Tf$= ${_R\Tf}$ of torsion-free left $R$-modules forms a definable subcategory of $\RMod$ if and only if $R$  is right P-coherent, and, if that is the case, $\Div_R$ and ${_R\Tf}$ are (elementarily) dual to each other (in the sense of Prest and Herzog), Theorem \ref{divtfD}.

This duality implies that, over right P-coherent rings, a left module is Mittag-Leffler with respect to divisible right modules, we say \emph{$\Div$-Mittag-Leffler}, if and only if it is what we call \emph{$\Tf$-atomic}, Corollary \ref{DivML}. $\Tf$-atomicity is a certain finiteness condition on the  finite systems of linear equations that have solutions in the given module; see Sections \ref{F-at} and \ref{KML} for  definitions.  

The final goal of the paper is to characterize $\Div$-Mittag-Leffler abelian groups (and modules over more general rings). We prove that $M$ is  $\Div$-Mittag-Leffler if and only if $M/\T(M)$ is  Mittag-Leffler if and only if it is $\aleph_1$-free, where $\T(M)$ is the torsion part of $M$, Corollary \ref{AG}.
Here is how.

The break-down into  torsion  and  torsion-free cases works for abelian groups, since the torsion part is always a pure subgroup and any class of relativized Mittag-Leffler modules is closed under pure extension---as a matter of fact, over any ring. For the converse, we use that these classes are also closed under pure submodule. What is less straightforward  is whether a class of  relativized Mittag-Leffler modules is closed under taking the largest torsion-free factor module. In other words, if $M$ is, say  $\Div$-Mittag-Leffler, is so also $M/\T(M)$?
  
We pursue this question over a large class of rings for which we are guaranteed that $\T(M)$ is a pure submodule of $M$, namely over semihereditary RD-rings. These are semihereditary rings over which RD-purity (=relative divisibility) is the same as full purity, see below. By  Hattori's work, semiheredity guarantees that $\Tf$ is a torsion-free class of a torsion theory, so that   $M/\T(M)$ is always torsion-free and therefore $\T(M)$  always RD-pure in $M$ (in fact, less suffices,  but for RD-rings there is no difference). The RD property then entails that  $\T(M)$ is, in fact, pure in $M$. 

We will not have to worry about the torsion case much, as, trivially, torsion modules are  $\Div$-Mittag-Leffler, Proposition \ref{torsionisML}---for the simple reason that a divisible module tensored with a torsion module is zero, a fact already noticed by Hattori.

Here is but a few examples of RD-rings: $\mathbb{Z}$, $k[X]$ for any field $k$, serial rings, Pr\"ufer rings, von Neumann regular rings, and the so-called Dedekind prime rings, whose prominent example is the first Weyl algebra $A_1(k)$ over any field $k$ of characteristic $0$.

Recall that over $\mathbb{Z}$, `divisible' is the same as `injective' and as `absolutely pure' (or `fp-injective'). Similarly, `torsion-free' is the same as `flat.'
Denote the class of \emph{flat} left $R$-modules by $\flat$ (or, to be precise, ${_R\flat}$) and that of \emph{absolutely pure} right $R$-modules by $\sharp$ (or $\sharp_R$). We may then restate the aforementioned Corollary \ref{DivML} for $R=\mathbb{Z}$ thus: an abelian group is $\sharp$-Mittag-Leffler if and only if it is $\flat$-atomic. This is actually true over \emph{any} ring, see Fact \ref{MThm}(1).

Over RD-rings the same holds true:  $\Div=\sharp$ and $\Tf=\flat$, and it doesn't matter on which side, as the RD property is two-sided. So we may as well consider $\sharp$-Mittag-Leffler   modules. It  turns out that such modules are already (fully) Mittag-Leffler---provided they are flat, Theorem \ref{alreadyML}. Consequently, over the rings in question, we obtain the final description that $M$ is $\sharp$- (or $\Div$-) Mittag-Leffler if and only if so is $M/\T(M)$, Theorem \ref{thmRD}. In case of  (not necessarily commutative) domains, the RD property entails semi-heredity, 
so we may drop that requirement in our final Theorem \ref{thmRDOre} that describes $\sharp$-Mittag-Leffler modules over RD-Ore domains in terms of almost projectivity of their largest torsion-free factor.

We draw conclusions about the special cases of (not necessarily commutative) Dedekind domains (including the first Weyl algebra, see Remark \ref{3}(1)), B\'ezout domains that are RD and, in particular, one-sided PID's that are RD, Corollary \ref{corBez}, where almost projectivity becomes almost freeness.

We actually restrict our search to  rings over which the torsion theory is hereditary, i.e., over which submodules of torsion modules are torsion. This is the case over RD-domains that are in addition Ore, Lemma \ref{T=t}, for  then, as already noticed by Hattori, we are back in the classical torsion theory. This restriction still includes all noetherian RD-domains and more, while, as a final note, it is not even known if there is any RD-domain that is \emph{not} Ore. 

On the way, in Section \ref{ML}, we collect some general results on flat Mittag-Leffler modules over semi-hereditary rings and, in particular, over von Neumann regular rings and semi-firs, that may be of interest in their own right.

\section{Background}

\subsection{Elementary Duality}\label{elemdual}
We assume the reader familiar with  \emph{elementary duality} as defined by Prest \cite{P}. It constitutes, for a given ring, an anti-isomorphism $\D$ between the lattices of (strictly speaking, equivalence classes of) pp formulas on either side of the ring. We use the same symbol $\D$, no matter which side we apply it to (thus $\D^2=1$ makes sense). The easiest cases are: $D(rx=0)$ is equivalent to the formula $r|x$ (thought of  on the right, i.e., as $\exists y(x=yr)$) and $\D(r|x)$ is equivalent to $rx=0$. On the other side we have that  $\D(xr=0)$ is (equivalent to) $\exists y(x=ry)$ and $\D(\exists y(x=ry))$ is $xr=0$. For the general case and other properties, see  \cite[Sect.1.3]{P2} or \cite{Pre}---or \cite[Ch.1]{habil}, which may serve as a swift introduction to most of the necessary model-theoretic background.

From the anti-isomorphism it is obvious how to extend $\D$ to implications of pp formulas, i.e.,  statements of the form $\forall x (\phi\to \psi)$, where $\phi$ and $\psi$ are pp formulas in the same free variable $x$ (or a tuple of such). Namely, the dual of this implication is declared to be $\forall x (\D\psi\to \D\phi)$. Thinking within the lattice, we often  use $\leq$ instead of the implicational arrow.

Herzog  \cite{Her} showed (among other, more general, things) that a collection $\Sigma$ of such implications has a model if and only if its \emph{dual}  $\D  \Sigma$ does, by which we mean the collection of all the duals of implications in $\Sigma$. An elegant proof of this can be given using character duals, see   \cite{ZZ-H}, \cite{PRZ2},  or \cite[Sect.1.3.3]{P2}, where it is shown that a module $M$ satisfies all the statements in $\Sigma$ if and only if its \emph{character module} 
$M^*=\Hom_\mathbb{Z}(M, \mathbb{Q}/\mathbb{Z})$ satisfies all the statements in $\D\Sigma$.

\subsection{Definable subcategories}
A \emph{definable subcategory}  is a full subcategory of the category of all modules over a given ring that is closed under pure submodules, direct limits and direct products. These are known to be exactly the classes of modules \emph{definable} (i.e., axiomatizable) by sets of implications of (unary) pp formulas as above \cite[Thm.3.4.7]{P2}. For properties and examples see Section \ref{Ivo} and Theorem \ref{divtfD} below and, especially, \cite[Sect.3.4]{P2}. 

We use $\phi\leq_\cal X \psi$ to mean that $\phi(X)\subseteq\psi(X)$ for all $X\in\cal X$; and $\sim_\cal X$ stands for `$\leq_\cal X$ and $\geq_\cal X$.' Then, clearly, a definable subcategory $\cal X$ is axiomatized by the set of all  equivalences  $\forall x(\phi\leftrightarrow \psi)$ that are  true in $\cal X$, i.e., for which $\phi\sim_\cal X\psi$.  And if $\cal X$ wasn't a definable subcategory to begin with,  this set axiomatizes the \emph{definable subcategory generated by $\cal X$}, denoted  $\langle \cal X \rangle$, i.e., the  smallest definable subcategory containing $\cal X$. Note, the relations $\leq_\cal X$ and $\leq_{\langle \cal X \rangle}$ are the same.

It may be most instructive to understand definable subcategories of, say $\RMod$ as the closed sets (or rather classes) of the Galois correspondence between left $R$-modules and implications of pp formulas true in them. More precisely, let $\RImp$  denote the set of all implications of left 1-place pp formulas, i.e., sentences of the form $\forall x (\phi \to \psi)$, where we assume, without loss of generality, that $\forall x (\psi \to \phi)$ is true in any left $R$-module.

Consider the correspondence between (subclasses of)  $\RMod$  and (subsets of)   $\RImp$ 
given by
$$\cal X \mapsto \Sigma({\cal X}) := \{ \sigma\in\RImp : \sigma \mbox{ is true in every } X\in \cal X\},$$
$$\Sigma \mapsto \cal X({\Sigma}) := \{X\in\RMod :   \mbox{ every  } \sigma \in \Sigma \mbox{ is true in  } X  \}.\footnote{Logically speaking, $ \Sigma({\cal X})$ is the implicational theory of $\cal X$, while $\cal X({\Sigma})$ is the model class of $\Sigma$, cf.\ \cite[Ex.3.4.4]{R}.}$$
  
It is easy to check that this is indeed a Galois connection, whose closed subclasses of  $\RMod$ are the definable subcategories.\footnote{The closed subsets of  $\RImp$ are the deductively closed subsets, i.e., the implicational theories of classes of modules.} Further, $\langle \cal X \rangle = \cal X(\Sigma(\cal X))$ and $\Sigma(\cal X)=\Sigma(\langle \cal X \rangle)$ and therefore
$\cal Y\subseteq \cal X \Rightarrow \langle \cal Y \rangle\subseteq \langle \cal X \rangle \Leftrightarrow \Sigma(\cal X)\subseteq\Sigma(\cal Y).$ But $\leq_\cal X$ is the set of all pairs $(\phi, \psi)$ with  $\forall x (\phi \to \psi)\in \Sigma(\cal X)$, so the right hand side of the previous equivalence is the same as saying that the relation $\leq_\cal X$ is contained in the relation $\leq_\cal Y$. We conclude:

\begin{fact}\label{inclusion} Let $\cal  X$ and $\cal Y$ be classes of modules (on the same side).

$ \langle \cal Y \rangle\subseteq \langle \cal X \rangle$ if and only if   $\phi \leq_\cal X \psi$ implies  $\phi \leq_\cal Y \psi$ for all (1-place) pp formulas $\phi$ and $\psi$.
\end{fact}

The \emph{dual}, $\D\cal X$, of a definable subcategory $\cal X$ is by definition the definable subcategory (on the other side) given by those axioms $\forall x(\D\psi\rightarrow \D\phi)$ for which $\phi\leq_{\cal X}\psi$ (in other words, one just dualizes all implications that hold in $\cal X$). Clearly, $\phi\leq_{\cal X}\psi$ if and only if $\D\psi\leq_{\D\cal X}\D\phi$. By Section \ref{elemdual}, a module $X$ is in $\cal X$ if and only if the character dual $X^*$ is in $\D\cal X$.

\subsection{Extensions of elementary duality}

In \cite{PRZ2} elementary duality was extended to certain infinitary implications, most prominently (but not only) to sentences of the following forms. 

An \emph{A-sentence} is a sentence of the form 
$\forall x \{(\bigwedge_{i} \psi_i) \to \phi\}$ with $\phi$ and the $\psi_i$ pp formulas (and the conjunction possibly infinitary), 
while an \emph{F-sentence} is one of the form 
$\forall x \{\psi\to (\sum_{i} \phi_i)\}$ with $\psi$ and the $\phi_i$ pp formulas (and the sum possibly infinitary).
The subformulas $\phi, \psi$ etc are assumed to be  in the same free variable $x$ (which may in general be a finite collection, but for all purposes at hand we may restrict ourselves to the case of a single variable).

Here \emph{A} stands for ``absolutely pure" and $F$ for ``flat," the reason for which will become clear in the next section.

Elementary duality is extended to A- and F-sentences as follows. The possibly infinitary sentences $\forall x \{(\bigwedge_{i} \psi_i) \to \phi\}$
and
$\forall x \{\psi\to (\sum_{i} \phi_i)\}$ are defined to be \emph{dual} if so are the pp formulas $\phi$ and $\psi$, as well as the pp formulas $\phi_i$ and $\psi_i$, for all $i$. 
In other words, the dual of the A-sentence $\forall x \{(\bigwedge_{i} \psi_i) \to \phi\}$ is the F-sentence $\forall x \{\D\phi\to (\sum_{i} \D\psi_i)\}$.
Notice, this makes the latter statement one for right modules if the former was one for left modules.
Similarly, the dual of the F-sentence $\forall x \{\psi\to (\sum_{i} \phi_i)\}$ is the A-sentence $\forall x \{(\bigwedge_{i} \D\phi_i) \to \D\psi\}$. (Again, we do not need to specify sides, when we, as usually, let $\D$ act both ways, left to right, as well as right to left, so that $\D^2$ is the identity.) The main tool from \cite[Thm.4.3]{PRZ2}
now states as follows.

\begin{fact}\label{PRZ}
\begin{enumerate}[\upshape(1)]
\item An F-sentence is true in a module $M$ if and only if its dual is true in $M^*=\Hom_\mathbb{Z}(M, \mathbb{Q}/\mathbb{Z})$. 
\item An A-sentence is true in $M$ whenever its dual is true in $M^*$. 
\end{enumerate}
\end{fact}

It  follows from W\"urfel's theorem (see the next section) that the converse of (2) is not true.

\subsection{Absolutely pure and flat modules}\label{Ivo}
Denote the class of \emph{flat} left $R$-modules by $\flat$ (or, to be precise, ${_R\flat}$) and that of \emph{absolutely pure} (or \emph{fp-injective}) right $R$-modules by $\sharp$ (or $\sharp_R$). Recall, over right noetherian rings, absolutely pure right modules are injective.

First we have an important characterization of flatness  \cite[1.3(a)]{Zim}, see also  \cite[Thm.2.3.9]{P2}.

\begin{fact}[Zimmermann] \label{flat} The following are equivalent for any left module $M$ over an arbitrary ring $R$.
\begin{enumerate}  [\upshape (i)]
 \item $M$ is flat.
 \item $\phi(M)=\phi({_R}R)M$ for all (unary) pp formulas $\phi$ (for left $R$-modules).
 \item $M$ satisfies all F-sentences of the form $\forall x \{\phi(x)\to (\sum_{s\in\phi({_R}R)} s|x)\}$, where $\phi$ is a unary pp formula.
\end{enumerate}
\end{fact}

Both conditions hold also for many-place pp formulas $\phi$ \cite[Thm.2.3.9]{P2}. Notice, if the right ideal  $\phi({_R}R)$ is not finitely generated, the conclusion involves an infinite sum (or, strictly speaking, disjunction), hence the F-statement would be  infinitary in that case.

Dually,   \cite[Prop. 1.3]{PRZ1} yields a similar characterization of absolute purity \cite[Prop.2.3.3]{P2}.

\begin{fact}\label{abspure} The following are equivalent for any right module $N$ over an arbitrary ring $R$.
\begin{enumerate}  [\upshape (i)]
 \item $N$ is absolutely pure.
 \item $\phi(N)=\ann_N\D\phi({_R}R)$ for all (unary) pp formulas $\phi$  (for right $R$-modules).
 \item $N$ satisfies all A-sentences of the form 
 $\forall x \{(\bigwedge_{s\in\D\phi({_R}R)} xs=0) \to \phi(x)\},$
 where $\phi$ is a unary pp formula.
\end{enumerate}
\end{fact}

 Here $\D\phi({_R}R)$ is the right ideal defined by $\D\phi$, a `left' formula, in the module ${_R}R$. (Again, both conditions hold equivalently for many-place pp formulas, but \cite[Prop.2.1.6]{P2} allows us to restrict to unary formulas $\phi$ also here.) When the right ideal $\D\phi({_R}R)$ is not finitely generated, the antecedent of that A-sentence involves an infinite conjunction, hence it would be  infinitary in that case.

In view of conditions (iii) of the previous two descriptions,  Fact \ref{PRZ} above gives us the following special case.

\begin{fact}\label{LambWurf}
\begin{enumerate}[\upshape(1)]
 \item {\rm (Lambek).} $M$ is flat if and only if  $M^*$  is absolutely pure.
 \item  {\rm (W\"urfel).}  $M$ is absolutely pure whenever  $M^*$  is flat.
\end{enumerate}
\end{fact}

That the converse of (2) is not true follows from W\"urfel's theorem saying that the character module  of every absolutely pure left $R$-module is flat if and only if $R$ is left coherent. It is this result (and  the simple proof it was given in \cite[Thm.4.4]{PRZ2}) that inspired the search for the characterization given in Theorem \ref{divtfD} below. 
The next, classical, result is the model for that theorem, see \cite[Thm.3.4.24]{P2} (or \cite{P}) for references and proof.

\begin{fact}[Eklof--Sabbagh]\label{ES} 
\begin{enumerate}[\upshape(1)]
\item
 The following are equivalent for any ring $R$.
\begin{enumerate}  [\upshape (i)]
 \item  $\flat = {_R\flat}$ is a definable subcategory. 
 \item $\sharp = \sharp_R$  is a definable subcategory. 
 \item $R$ is right coherent.
\end{enumerate}
\item In this case, $\flat$ and $\sharp$ are dual to each other.
\end{enumerate}
\end{fact}

So, if the ring  is right coherent, the duality of the A- and F-sentences exhibited in  clauses (iii) of Facts \ref{flat} and \ref{abspure} readily implies the duality of the corresponding definable categories, for then those sentences are, in fact, finitary (because then the corresponding conjunctions and sums are) \cite[Prop.3.4.24]{P2}.  When  the ring  is not right coherent (and $\flat$ and $\sharp$ no longer constitute definable subcategories),  the control one has over all pp subgroups in flat and absolutely pure modules thanks to clauses (ii) in those facts above  (which is what seems to be missing for divisible and torsion-free modules!), allows one, on the other hand, to prove the same for the corresponding generated definable subcategories  \cite[Cor.12.2]{Her}, see also \cite[Prop.3.4.26]{P2}.

\begin{fact}[Herzog] \label{IvoH} 
 The dual of the definable category generated by  $\flat$ ($ = {_R\flat}$) is the definable category generated by $\sharp$ ($ = \sharp_R$, over any ring $R$).
\end{fact}
 
\subsection{Almost projective modules}\label{cd} 
The theme of \cite{EM} is various approximations to freeness of abelian groups and modules. We are mostly interested in the same thing for projectivity, but the terminology tends to be inconsistent or uneven at the least.  Therefore we first fix the notation.

Let  $\kappa$ be a cardinal, $\cal S$ a property of modules, and $\cal F$ a class of modules for which the term `$\cal F$-pure' makes sense. In the way we use this, `$\emptyset$-pure submodule' will mean `submodule' and `($\RMod$)-pure submodule' will mean `pure submodule.'

Sticking to the original usage, we say a module is  $\kappa$-$\cal S$
  if every \emph{$<\kappa$-generated}  submodule (i.e., submodule generated by less than $\kappa$ elements) has property $\cal S$. We need various versions of this with some sort of covering by $\cal F$-pure submodules having property $\cal S$  involved (so far, `$\cal F$-pure' has not yet been given any meaning, which it will be in Definition \ref{F-pure}). Call a module  $\kappa$-$\cal F$-$\cal S$ if every \emph{$<\kappa$-generated}  submodule (or simply every subset of power $<\kappa$) is contained in a $<\kappa$-generated $\cal F$-pure submodule that has property $\cal S$.  (This is not properly a weakening, for it involves an existence statement.) 

We write $\kappa$-$*$-$\cal S$ instead of $\kappa$-($\RMod$)-$\cal S$ and $\kappa$-$\circ$-$\cal S$ instead of $\kappa$-$\emptyset$-$\cal S$. In other words, a module is $\kappa$-$\circ$-$\cal S$
($\kappa$-$*$-$\cal S$) if every  $<\kappa$-generated submodule is contained in a $<\kappa$-generated (pure) submodule with property $\cal S$.

 Trivially $\kappa-\cal S$ implies $\lambda-\cal S$ for all $\lambda\leq \kappa$. But for the $\cal F$-adorned concept this may not be true, simply for the lack of $\cal F$-pure submodules generated by few enough elements (which is no problem when $\kappa > |R|$, as then every subset of power $\lambda<\kappa$ is contained  in a pure submodule of power $\lambda+|R|<\kappa$).
 
 One may think of a certain  class of $<\kappa$-generated $\cal F$-pure submodules with property $\cal S$ as a covering class $\cal C$ for subsets of power $<\kappa$ in the above ($\cal C$ need not  be the class of \emph{all} such submodules!). Dropping any purity entirely and adding a continuity condition on $\cal C$ instead, we arrive at the following definition. 
 
 We say a module $M$ is 
  $\kappa$-c-$\cal S$ 
  (`c' for continuity and covering) if there is a set $\cal C$, call it a $\kappa$-c-$\cal S$ \emph{covering of $M$}, of $<\kappa$-generated submodules  of $M$ with property $\cal S$ such that
  
\mbox{}\hspace{10pt}  (d){\tiny ensity} every subset of $M$ of power $<\kappa$ is contained in a member of $\cal C$,
  
 \mbox{}\hspace{10pt}  (c){\tiny ontinuity} $\cal C$ is closed under unions of chains\footnote{Some authors prefer to stress \emph{well-ordered}, but of course this is redundant: one can always find a cofinal well-ordered chain 
whose union is the same.}
 of length $<\kappa$.

If $\kappa=\aleph_1$, we may confine ourselves to chains of order type $\omega$ in condition (c), for every countable ordinal has cofinality $\omega$.

In \cite{EM}, for $\cal S$ the class of free modules, this is called `$\kappa$-free,' while $\kappa$-$\circ$-free is called  `$\kappa$-free in the weak sense'  \cite[Def.IV.1.1, p.83, and p.84]{EM}.
 Notice, $\kappa$-$\circ$-$\cal S$ is as $\kappa$-$c$-$\cal S$ without the continuity condition (c). 

In the proof of Theorem \ref{alreadyML}   for  uncountable rings we will need the following result from \cite{BT}. (The authors' terminology is different: what we call $\kappa$-c-projective, they simply call `$\kappa$-projective' (following the template of \cite[Def.IV.1.1]{EM}),  while they call `weakly  $\kappa$-projective' what we call $\kappa$-$*$-projective.)

\begin{fact}\label{BT}\mbox{}\cite[Lemma 1.3]{BT}. A module is flat and Mittag-Leffler if and only if it is 
  $\aleph_1$-c-projective. 
  A left module over a left hereditary ring is flat Mittag-Leffler if and only if it is $\aleph_1$-projective.
\end{fact}

This description of flat Mittag-Leffler modules, over \emph{countable} left hereditary rings (which applies, note,  to the first Weyl algebra over a field of characteristic 0 \cite[Exples.2.32(h)]{L}), was  contained  in \cite[Cor.6.5]{habil}.

Let us isolate the general half of the last statement for further reference. It follows from the simple observation that, in an $\aleph_1$-projective module, the set of all countably generated submodules constitutes an $\aleph_1$-c-projective covering. (For countable rings the result  was contained in \cite[Fact 6.4]{habil}, for strongly non-singular semi-hereditary Goldie rings of arbitrary cardinality in \cite[Prop.9]{AF'}.)

\begin{cor}\label{aleph1}
Over any ring,  $\aleph_1$-projective modules  are flat and Mittag-Leffler (in particular,  $\aleph_1$-free modules are).
\end{cor}

\subsection{Some ring-theoretic properties}\label{ring}
Following \cite[p.151]{Hat}, call a ring \emph{left PP} (resp.\  \emph{PF}) if every principal left ideal is projective (resp.\ flat). Following \cite[Def.2.1]{MD}, call a ring  \emph{left P-coherent} if every principal left ideal is finitely presented or, equivalently, the left annihilator of any ring element is finitely generated.\footnote{It was pointed out in \cite[before Lemma 4.10]{MD} that Nicholson's `left morphic rings' are left P-coherent.} As `finitely presented + flat = finitely generated projective' (which was also proved in \cite[Lemma 1]{Hat}), we see that `PF + P-coherent = PP.'

A \emph{Dedekind prime ring} is a two-sided hereditary and noetherian prime ring without idempotent two-sided ideals \cite[5.2.10+5.6.3]{MR}.  A \emph{Dedekind domain} is a  two-sided  hereditary  and noetherian domain without idempotent   two-sided ideals  \cite[5.2.11]{MR}.  (Throughout,  \emph{domain} is used to mean, not necessarily commutative, `ring without zero-divisors.')

We will make  repeated use of some classical results on projective modules over semi-hereditary rings. Albrecht  \cite{Alb} showed that every left projective over a left semi-hereditary ring is a direct sum of   finitely generated left ideals, see also  \cite[39.13(2)]{W}. Bass extended this to the other side, by showing that right projectives are direct sums of duals (by homing into the ring) of  finitely generated left ideals \cite{Bas}, see also \cite[Comments, Ch.0]{C} or  \cite[Comments, Ch.1]{C'}. Thus, all of these are free if all finitely generated left ideals are free, which leads to another kind of ring. Given a cardinal number $\lambda$, a left \emph{$\lambda$-fir} is a ring in which every $\leq\lambda$-generated left ideal is free of unique rank \cite[Sect.1.2]{C} or  \cite[Sect.2.2]{C'}. For finite $\lambda$, this turns out to be symmetric, \cite[Thm.1.1.3]{C} or \cite[Thm.2.3.1]{C'}, while for infinite $\lambda$ it is not,  \cite[Sect.1.2]{C} or  \cite[Thm.2.10.3 and thereafter]{C'}. A \emph{semi-fir} is a ring that is an $n$-fir for every natural number $n$. Clearly, this too is  left-right symmetric. Note, a $1$-fir is simply a domain, so all of these rings are domains. Thus, semi-firs are two-sided semi-hereditary domains, while left firs are left hereditary domains. 

Obviously, a left principal ideal domain is a left fir, while a left or right  B\'ezout domain is a semi-fir. In fact, a ring is a left  B\'ezout domain if and only if it is a $2$-fir satisfying the left Ore condition,  \cite[Prop.1.1.3]{C} or \cite[Prop.2.3.17]{C'}.

We quote for later reference, cf.\ \cite[Thm.2.3.11]{C'}.\footnote{Thank you, Mike!}

\begin{fact}[Albrecht/Bass]\label{Alb} 
 Every projective module (left  or right) over a left semi-hereditary ring decomposes into a direct sum of finitely generated projective modules. Over a semi-fir, all  projective modules are free.
 \end{fact}

All unexplained ring- or module-theoretic terminology can be found in \cite{C}, \cite{C'}, \cite{L}, \cite{S}, or \cite{W}.

\section{$\cal F$-atomic modules}\label{F-at} 

The significance of $\cal F$-atomic modules lies in the fact that they are exactly the $\cal K$-Mittag-Leffler modules for   $\cal K=\D\langle\cal F\rangle$, the definable subcategory dual (with respect to elementary duality) to the definable subcategory generated by $\cal F$, cf.\  Fact \ref{MThm}(1) below. 

A module is said to be \emph{$\cal F$-atomic}\footnote{Beware, the usage of this term in \cite{habil} is slightly different.} if every finite tuple in it has its pp type $\cal F$-generated by some pp formula---we also say, all pp types realized are \emph{$\cal F$-finitely generated}, i.e., every given tuple $\bar{a}$ in the  $\cal F$-atomic module $A$ satisfies a pp formula $\phi$ that  $\cal F$-implies any other formula  $\bar{a}$ may satisfy in $A$. The latter means that $\phi\leq_{\cal F} \psi$ for every $\psi$ in the pp type of $\bar{a}$. As before,  $\leq_{\cal F}$ is the partial ordering in the lattice of pp formulas restricted to modules in $\cal F$, i.e., $\phi\leq_{\cal F} \psi$ if and only if $\phi(F)\subseteq \psi(F)$ for every $F\in \cal F$.

\subsection{Free realizations}
 $\cal F$-atomic modules, at least in the countably generated case, can be conveniently mapped to modules in $\cal F$ provided they share some pp formula behavior---just like  free realizations do.
Free realizations of formulas were introduced by Prest as a generalization of presentations of structures by generators and relations where  the relations can appear in the form of a pp formula (hence, with existential quantifiers involved) \cite{P}, \cite{P2}. 

We need only one instance of this phenomenon, namely for $\cal F=\Tf$, the class of torsion-free modules. However, the proofs are just the same, so we keep the generality for later reference.
It should be mentioned  that these results, and especially their proofs, go in essence back to \cite{habil}.

\begin{lem}
Suppose  $A$ is an $\cal F$-atomic module generated by $a_0, a_1, \ldots$. Let, for all $i$, the pp type of $(a_0, \ldots, a_i)$ be $\cal F$-generated by the pp formula $\phi_i=\phi_i(x_0, \ldots, x_i)$.\footnote{Note, this alone suffices to make $A$ an $\cal F$-atomic module.} Then $A$ is a \emph{free realization of the sequence $\phi_0, \phi_1, \ldots$ for $\cal F$} in the sense that, whenever a countable sequence $c_0, c_1, \ldots$ in a module $F\in \cal F$ is such that, for all $i<\omega$, the tuple $(c_0, \ldots, c_i)$ satisfies $\phi_i$, then there is a map $A\to F$ sending $a_i$ to $c_i$ for all $i$.
\end{lem}
\begin{proof}
All we need to do is verify that the map indicated is well defined. So let $\sum_{j\leq i} r_j a_j =0$ in $A$ for some ring elements $r_j$. Then the pp formula $\sum_{j\leq i} r_j x_j =0$ (an equation in this case) is satisfied by the tuple $(a_0, \ldots, a_i)$ in $A$ and therefore $\cal F$-implied by $\phi_i$. In particular, $\phi_i$ implies $\sum_{j\leq i} r_j x_j =0$ in $F$. Since, in $F$,  the tuple $(c_0, \ldots, c_i)$ satisfies the former, it does satisfy also the latter, hence $\sum_{j\leq i} r_j c_j =0$, as desired.
\end{proof}

\begin{prop}\label{freereal}
 For every tuple $\bar{a}$ in  a countably generated $\cal F$-atomic module  $A$, there is a pp formula $\phi$ that \emph{$(F, \bar{a})$ freely realizes for $\cal F$} in the sense that, whenever a tuple $\bar{c}$ (of the same length) satisfies $\phi$ in a module $F\in \cal F$, then there is a map  $A\to F$ sending $\bar{a}$ to $\bar{c}$.
\end{prop}
\begin{proof}
Let $\bar{a}=(a_0, \ldots, a_k)$ and write the generators in question as $a_{k+1}, a_{k+2}, \ldots$. 
By $\cal F$-atomicity, we choose $(i+1)$-place pp formulas $\phi_i$ as in the lemma, for all $i<\omega$. As the pp formula $\exists x_{i} \phi_{i}$ is satisfied by the $i$-tuple $(a_0, \ldots, a_{i-1})$,  whose pp type is $\cal F$-generated by $\phi_{i-1}$, we have $\phi_{i-1}\leq_{\cal F}\exists x_{i} \phi_{i}$ for all $i$. Consequently,  the formulas $\exists x_{k+1} x_{k+2} \ldots  x_{i}\phi_{i}$,  for all $i>k$, are $\cal F$-implied by  $\phi:=\phi_k$. 

This allows us to choose a sequence $c_{k+1}, c_{k+2}, \ldots$ in $F$ in such a way that $c_0, c_1, \ldots$ plays the same role in $F$ as it does in the lemma, with the extra information that it starts with the tuple $\bar{c}$. Therefore, the map the lemma yields maps $\bar{a}=(a_0, \ldots, a_k)$ to $\bar{c}=(c_0, \ldots, c_k)$ (entry by entry), just as claimed.
\end{proof}

Note, by no means has the module $A$ itself to be in $\cal F$.

\subsection{Pure projectivity}
We give a similar application of the above lemma showing that countably generated  $\cal F$-atomic modules are projective with respect to pure epimorphisms emanating from members of  $\cal F$. This result is a variant of a generalization given in \cite[Lemma 3.9]{habil}  of Raynaud and Gruson's well known result  that countably generated Mittag-Leffler modules are pure-projective---see \cite[Prop.1.3.26]{P2}, whose model-theoretic proof is the one given in \cite[Lemma 3.9]{habil}, reproduced here with the obvious adjustments.

\begin{prop}\label{Fpureproj}
Every pure epimorphism from a module from $\cal F$ onto a countably generated $\cal F$-atomic module splits.
\end{prop}
\begin{proof} Let $g: F\to A$ be a pure epimorphism with  $F\in\cal F$ and $A$ countably generated $\cal F$-atomic. Choose generators $a_i$ of $A$ and pp formulas $\phi_i$ as in the hypothesis of the lemma. Note that, as in the proof of the previous proposition, this yields 
$$\phi_{i}\leq_{\cal F}\exists x_{i+1} \phi_{i+1}$$ 
for all $i$.

To invoke the lemma in order to obtain a splitting map of $g$, we need to choose $g$-preimages $c_i$ of $a_i$ in $F$ so that the resulting sequence $c_0, c_1, \ldots$ satisfies the formulas $\phi_i$ just as $F$ does in the statement of the lemma.

For $c_0$ we take any $g$-preimage of $a_0$ in
$\phi_0(F)$\,.
(We use purity without ado to pick preimages in the same pp subgroups.)  Having chosen a preimage $\bar{c}_i$ of $\bar{a}_i$ in
$\phi_i(F)$\,, apply the implication displayed above to obtain $b_{i+1}$ satisfying
$\phi_{i+1}(\bar{c}_i,z)$ in $F$ and then any preimage $b'_{i+1}$ of
$g(b_{i+1})-a_{i+1}$ in $\phi_{i+1}(\bar{0},F)$\,. (This is possible,
since, by additivity of pp formulas,  \mbox{$F\models \phi_{i+1}(\bar{c}_i,b_{i+1})$} implies
$A\models\phi_{i+1}(\bar{a}_i,g(b_{i+1}))$\, which, together with
$A\models\phi_{i+1}(\bar{a}_i,a_{i+1})$ yields
$A\models\phi_{i+1}(\bar{0},g(b_{i+1})-a_{i+1})$\,.) Then, setting
$c_{i+1}=b_{i+1}-b'_{i+1}$\,, we have $g(c_{i+1})=a_{i+1}$ and
$F\models\phi_{i+1}(\bar{c}_i,c_{i+1})$ by additivity again.
\end{proof}

By Fact \ref{MThm}(1) below, setting $\cal F=\RMod$ (or $\cal K=\ModR$) this yields at once

\begin{cor}[Raynaud--Gruson]\label{RGc.g.}\mbox{}
Every countably generated Mittag-Leffler module is pure-projective. (Hence every  countably generated flat Mittag-Leffler module is projective.)
\end{cor}

\section{Mittag-Leffler modules}\label{ML}
\subsection{Relativized Mittag-Leffler modules}\label{KML}
Suppose ${\cal K}$ is a family of right $R$-modules. Following \cite[Sect.2.2]{habil}, a \emph{${\cal
K}$-Mittag-Leffler module} 
 is
a left $R$-module $M$ for which the canonical map
$(\prod_I N_i) \otimes M\longrightarrow\prod_I (N_i\otimes M)$
is a monomorphism for all subfamilies $\{N_i : i\in I\}$ of ${\cal K}$\,. In case  ${\cal K}=\ModR$ these are the  \emph{Mittag-Leffler modules} introduced by Raynaud and Gruson in \cite{RG}. Another precursor is Goodearl's \cite{Goo}, which investigates, under different (actually no specific) name, the case  ${\cal K}=\flat_R$.

 Note that the hypothesis of (1) in the next fact simply means that $\phi\leq_{\cal F} \psi$ if and only if $\D\psi\leq_{\cal K} \D\phi$. Since $\flat_R$ and $\{R_R\}$ generate the same definable subcategory (by Fact \ref{flat}), as a consequence of (1) one also obtains Goodearl's result that $\flat_R$-Mittag-Leffler and $R_R$-Mittag-Leffler are the same \cite[Thm.1]{Goo}.

\begin{fact} \label{MThm} 
\mbox{}\cite[Main Thm.\ 2.2 and Cor.\ 2.4]{habil}{\rm , see also \cite{MLII}.}
\begin{enumerate}[\upshape(1)]
\item Suppose $\cal K$ and $\cal F$ are classes of, respectively, right and left modules over the same ring that generate mutually dual definable subcategories. 
 
 Then a left  module is ${\cal K}$-Mittag-Leffler if and only if it is $\cal F$-atomic.\\ In particular,    $\sharp$-Mittag-Leffler is the same as $\flat$-atomic and $\flat$-Mittag-Leffler is the same as $\sharp$-atomic.
 \item A module is ${\cal K}$-Mittag-Leffler if and only if every finite subset is contained in a pure submodule that is ${\cal K}$-Mittag-Leffler.
 \item The class  of  $\cal K$-Mittag-Leffler modules is closed under pure extensions and pure submodules.
\item A direct sum of modules is ${\cal K}$-Mittag-Leffler if and only if so is every direct summand.
\end{enumerate}
\end{fact}

It is not too hard to derive from this that if  $N$ is a finitely generated pure submodule of a $\cal K$-Mittag-Leffler module $M$, then  $M/N$ too is  $\cal K$-Mittag-Leffler \cite[Cor.2.4(e)]{habil}. This is no longer true for arbitrary pure submodules $N$, which is why we have to make an effort and restrict to certain rings in order to  prove that, if $M$ is $\cal K$-Mittag-Leffler,  so too is $M/\T(M)$ (with $\T(M)$ the torsion part).

The proofs of part (1) of the above fact (the others are immediate consequences) as published in \cite[Thm.2.2]{ML} and \cite[Thm.1.3.22]{P2} are executed 
only for the classical case $\cal K=\ModR$, they work, however, verbatim the same in general if  only $\leq$ in $\ModR$ gets replaced by $\leq_{\cal K}$ and in $\RMod$ by $\leq_{\cal F}$, as was done in \cite[Main Thm. 2.2]{habil}.

\subsection{Finding submodules}
The next fact goes back to the original paper of Raynaud and Gruson.

\begin{fact}\label{RG}
\mbox{}{\rm \cite[ Thm.2.2.1]{RG}, see also \cite[Thm.3.12+Rem.6.1]{habil} and \cite[Thm.1.3.27]{P2}.} 
\begin{enumerate}[\upshape(1)]
 \item  A module is Mittag-Leffler if and only if it is $\aleph_1$-$*$-pure-projective, i.e., iff every countable (finite suffices) subset is contained in a countably generated  pure-projective pure submodule.
 \item  A flat module is Mittag-Leffler if and only if it is $\aleph_1$-$*$-projective, i.e., iff every countable (finite) subset is contained in a countably generated projective pure submodule.
\end{enumerate}
\end{fact}

Of course, (2) follows from (1) utilizing the equation  flat+pure-projective = projective. (Curiously, it implies that projective modules are  $\aleph_1$-$*$-projective, which has been known since Kaplansky's theorem.)

A model-theoretic  proof of (1)  was given in \cite[Lemma 3.11]{habil}. A simplified version, \cite[Thm.1.3.27]{P2}, carries over one to one to the case of  $\cal K$-modules.  For the convenience of the reader I  indicate the proof, which is a standard algebraic/model-theoretic argument of adjoining solutions (as used, e.g., in order to produce algebraically closed fields/saturated models, see e.g.\ \cite[Sect.12.1]{R}). Purity has to be replaced however by the following weaker concept.

\begin{definition}\label{F-pure}
A submodule $N$ of an $\cal F$-atomic module $M$ is  \emph{$\cal F$-pure} if, for every tuple $\bar{a}$ in $N$, some $\cal F$-generator of  its pp type in $M$ is also contained in its pp type in $N$.
\end{definition}

 In other words,  there is an  $\cal F$-generator $\phi_{\bar{a}}$  of  the pp type of $\bar{a}$ in $M$ such that $\bar{a}\in \phi_{\bar{a}}(N)$, i.e., we demand purity only for the pp formulas that are $\cal F$-generators of pp types in $M$ that are realized in $N$.

Clearly, an $\cal F$-pure submodule of an $\cal F$-atomic module is  $\cal F$-atomic. We are ready to generalize Fact \ref{RG} to the relativized case. 

\begin{prop}\label{K-RG} Suppose $\cal K$ and $\cal F$ are classes of, respectively, right and left $R$-modules that generate mutually dual definable subcategories.

A module is  $\cal K$-Mittag Leffler if and only if every countable subset  is contained in a countably generated $\cal F$-pure submodule that is  $\cal K$-Mittag-Leffler.

Equivalently, a module is  $\cal F$-atomic  if and only if it is, in the terminology of Section \ref{cd},  $\aleph_1$-$\cal F$-`$\cal F$-atomic'.
\end{prop}
\begin{proof} 
In view of Fact \ref{MThm}(1), the two statements are indeed equivalent. We work with the latter.  So let $A$ be a countable subset of an $\cal F$-atomic module $M$. For each (of the countably many) finite tuples in $A$ adjoin witnesses to the $\cal F$-generating formula of its pp type in $M$ and denote the resulting set by $A'$. Repeat the process $\omega$ times and take the union, $B$ say, of the resulting ascending chain of subsets of $M$---a countable set. Let $N$ be the submodule of $M$ generated by $B$. If $\bar{a}$ is a tuple in $B$, it is contained in one of the sets on the way, hence the witnesses are contained in the next one, and certainly in $N$. Therefore, if $p$ is the pp type of $\bar{a}$  in $M$ and $q$ is that in $N$, then $q$ is contained\footnote{This is the only difference with  \cite[Thm.1.3.27]{P2}, where, because of $\cal F=\RMod$, we have in fact $p=q$, which yields (full) purity of $N$ in $M$.}  in $p$ and it  contains the  $\cal F$-generating formula of $p$.

We have verified $\cal F$-finite generation only for pp types of tuples $\bar{a}$ \emph{in} $B$, however, it was observed already in \cite[Lemma 1.5]{habil} that this implies the same for any tuple in the submodule \emph{generated} by $B$ (see also \cite[Fact 2.4]{PR}).
Consequently, $N$ itself is  $\cal F$-atomic. 
\end{proof}

\subsection{The flat case: finding flat submodules}
We now use the flatness criterion Fact \ref{flat} with no further mention. It turns out that here too it suffices to check it on generators.

\begin{lem}
Let $N$ be a module generated by a set $A$. For  $N$ to be flat it suffices that for every tuple $\bar{a}$ of generators from $A$ and every matching pp formula $\phi(\bar{x})$ it satisfies in $N$, the submodule of $R^{l(\bar{x})}_R$ defined by $\phi$ in $_RR$ divides  $\bar{a}$  in $N$, i.e., $\phi(_RR)|\bar{a}$ in $N$.
\end{lem}
\begin{proof}
Let $a=\sum_k s_k a_k$ with the $a_k$ in $A$. Write the $s_k$ as a tuple $\bar{s}$ (a row vector) and the $a_k$ as a tuple $\bar{a}$ (a column vector of the same length) and $a$ as $a=\bar{s} \, \bar{a}$. Consider a pp formula $\psi$ that $a$ satisfies in $N$. We need to show $\psi(_RR)|a$ (it is enough for flatness for this to be true for 1-place pp formulas $\psi$; the hypothesis on the generators we do require  for tuples though!).

Consider the pp formula $\phi(\bar{x})$ given by $\psi(\bar{s} \, \bar{x})$. Since $\bar{a}$ satisfies this formula in $N$, by hypothesis, there are  tuples $\bar{r}_i\in \phi(_RR)$ and elements $c_i\in N$ such that $\bar{a}= \sum_i\bar{r}_i \,  c_i$. But then $a=\bar{s} \, \bar{a}=\bar{s}( \sum_i\bar{r}_i  \,  c_i)= \sum_i (\bar{s} \,  \bar{r}_i ) c_i$. It remains to notice that each $\bar{s}  \, \bar{r}_i $ is in $\psi(_RR)$.
\end{proof}

\begin{prop}\label{flatK-RG}
Every countable subset of a flat $\sharp$-Mittag-Leffler  module is contained in a countably generated\/ $\flat$-pure submodule that is  flat and $\sharp$-Mittag-Leffler. That is, every  flat $\sharp$-Mittag-Leffler  module is $\aleph_1$-\,$\flat$-`flat  $\sharp$-Mittag-Leffler' (and conversely).
\end{prop}
\begin{proof}
 If $M$ is flat in the proof of Proposition \ref{K-RG}, we may work with certain particular  $\flat$-generators of types. Namely, let $\phi$ be a  $\flat$-generator of $p$. Let's work with 1-types first, so write $a$ instead of $\bar{a}$. Since $M$ is flat, $a$ satisfies a pp formula $\phi_a$ of the form $x=\exists\bar{y}\sum_i r_i y_i$ with $r_i\in\phi(_RR)$, i.e., $a=\sum_i r_i c_i$ for some $c_i\in M$. This formula $\phi_a$, of course, implies $\phi$ (everywhere), but being in  $p$, it is also $\flat$-implied by $\phi$. Consequently, it is  $\flat$-equivalent to $\phi$, and we may simply work with it instead of $\phi$. Then, in the terminology of the proof of the proposition, $N$ contains the witnesses $c_i$ and is $\flat$-atomic  (hence $\sharp$-Mittag-Leffler). To see that it is flat, consider (an arbitrary) $a\in N$ as before. Let $a\in \psi(N)$. We have to show $\psi(_RR)|a$ in $N$. As $\psi \in p$, the formula $\phi_a$ $\flat$-implies $\psi$. In particular, $\phi(_RR)=\phi_a(_RR)\subseteq\psi(_RR)$. Hence $r_i\in \psi(_RR)$, and thus the above representation of $a$ as $a=\sum_i r_i c_i$ shows that  $\psi(_RR)|a$ in $N$. 
 The obvious adjustments necessary for the treatment of tuples are left to the reader.
\end{proof}

\subsection{Flat $\sharp$-Mittag-Leffler modules are Mittag-Leffler}\label{sharp}

\begin{prop}\label{almostproj} Every countable subset of a  flat $\sharp$-Mittag-Leffler module is contained in a countably generated projective $\flat$-pure submodule. Hence,  a  flat module is $\sharp$-Mittag-Leffler if and only if it is $\aleph_1$-\,$\flat$-projective.
\end{prop}
\begin{proof} 
Let $C$ be a flat $\sharp$-Mittag-Leffler and choose an epimorphism $g: B \to C$ from a free module $B$. By flatness, $g$ is pure. Since $\sharp$-Mittag-Leffler is the same as $\flat$-atomic,
if $C$ is countably generated, Proposition \ref{Fpureproj} (for $\cal F=\flat$) yields  a section of $g$, showing that $C$ is a direct summand of $B$, hence projective.

By Proposition \ref{flatK-RG}, every countable subset of a flat $\sharp$-Mittag-Leffler is contained in a countably generated  $\flat$-pure  submodule which is flat and $\sharp$-Mittag-Leffler and hence, as just shown,  projective.
\end{proof}
 
  \begin{cor}
  \begin{enumerate}[\upshape(1)]
 \item A countably generated  flat module is $\sharp$-Mittag-Leffler if and only if it is projective.
 \item A   flat  left module over a left hereditary ring is $\sharp$-Mittag-Leffler if and only if it is $\aleph_1$-projective.
\end{enumerate}
 \end{cor}

\begin{thm}\label{alreadyML} 
 Every  flat $\sharp$-Mittag-Leffler module is  Mittag-Leffler.
\end{thm}
\begin{proof} Recall,  $\sharp$-Mittag-Leffler  is the same as $\flat$-atomic.

By Fact \ref{BT} we only need to find a covering class $\cal C$ of countably generated projective modules satisfying conditions (c) and (d) from Section \ref{cd} (in order to derive  $\aleph_1$-c-projectivity). If we take for $\cal C$ the class of  countably generated flat $\sharp$-Mittag-Leffler  $\flat$-pure submodules as constructed in the proof of Proposition \ref{flatK-RG}, we have (d) by the  proposition, and  every member of $\cal C$ is projective by the previous corollary.

In order to verify (c), consider a chain $C_0\subseteq C_1\subseteq C_2\subseteq \ldots$  in $\cal C$ (order type $\omega$ suffices) with union  $C$. We have to show $C$ is in $\cal C$. For that, in turn, it suffices to verify that  it contains all witnesses to the particular  $\flat$-generators of types of tuples from $C$  as employed in the proof of Proposition \ref{flatK-RG}. However, this is clear, since every such tuple is contained already in some $C_n$, which itself enjoys that property.
\end{proof}

Note that for countable rings the proof simplifies considerably, as then every countable set is contained in a countably generated \emph{pure} submodule that inherits all the properties in question. In the proof of Theorem \ref{alreadyML} we then have enough pure submodules that are projective, thus guaranteeing that the module in question is Mittag-Leffler  (by Fact \ref{MThm}(2)).

\subsection{Semi-hereditary rings}\label{s-hrings}
Now that we know flat $\sharp$-Mittag-Leffler modules  are  Mittag-Leffler (in the usual sense, i.e., with $\cal K=\RMod$), we can make use of some old results from  \cite[Ch.6]{habil} on flat Mittag-Leffler modules, which we record  here for reference---indicating the proofs if never  published elsewhere. 

We are going to derive some more equivalences to being flat and Mittag-Leffler. 

\begin{cond}\mbox{}\label{cond}
Consider the following conditions on a module $F$.

\begin{enumerate}  [\upshape  (A)]
\item    $F$ is flat and $\sharp$-Mittag-Leffler.
 \item  $F$ is flat and Mittag-Leffler.
 \item  $F$ is $\aleph_1$-$*$-projective.
 \item  $F$ is $\aleph_0$-$*$-projective.
 \item $F$ is $\aleph_1$-projective.
  \item $F$ is $\aleph_0$-projective.
\end{enumerate} 

Then {\rm (A) -- (C)} are equivalent and implied by any of {\rm (D) or (E)};  trivially, {\rm (E)} implies {\rm (F)}.
\end{cond}
\begin{proof}
 
We know already about the first three.
By Fact \ref{MThm}(2), given an infinite $\kappa$, every $\kappa$-$*$-`$\cal K$-Mittag-Leffler' module is $\cal K$-Mittag-Leffler. In particular, $\aleph_0$-$*$-pure-projective modules are Mittag-Leffler, and  $\aleph_0$-$*$-projectives are flat Mittag-Leffler.
For (E) see Corollary \ref{aleph1}.
\end{proof}

Next we characterize the rings over which flat Mittag-Leffler is the same as $\aleph_1$-projective. 
 They turn out to be the 
 \emph{left countably hereditary} rings, i.e., rings whose countably generated left ideals are projective. (In \cite[\S 2.1]{C'} these are called \emph{left $\aleph_0$-hereditary rings}.)

\cite[Ex.1.3.12]{C} (or \cite[Ex.2.3.9]{C'}) asks to show that a semi-fir is a left fir if and only if it is  left hereditary. The  definition just given allows one to add an intermediate layer (which gives away our solution to the exercise).
 
\begin{rem}\label{aleph0fir}
 A semi-fir is a left $\aleph_0$-fir if and only if it is left countably hereditary.
\end{rem}
\begin{proof}
 By definition,  a left $\aleph_0$-fir is a ring all of whose countably generated left ideals are free. But over a semi-fir, projective=free  by the last statement of Fact \ref{Alb}. 
 \end{proof}

If the ring itself is  countable 
 (or has only countably generated left ideals),  countable heredity is the same as  heredity. For that case
 the next characterization  was contained in \cite[Cor.6.6]{habil}. It implies that all von Neumann regular rings are countably hereditary, see Remark \ref{c-hered-reg} below. Property (ii)  was proved for regular rings in  \cite[Cor.6.8]{habil}.
 
\begin{prop}\label{c-hered} The following are equivalent over any ring.
\begin{enumerate}  [\upshape (i)]
 \item  $R$ is left countably hereditary, i.e.,  $_RR$ is $\aleph_1$-projective.
 \item Every projective left  $R$-module is  $\aleph_1$-projective.
 \item A   left  $R$-module is $\aleph_1$-$*$-projective if and only if it is $\aleph_1$-projective.
 \item A   left  $R$-module is flat and Mittag-Leffler  if and only if it is $\aleph_1$-projective.
 \item The property  of being  flat and Mittag-Leffler  is hereditary in $\RMod$  (i.e., inherited by submodules).
\end{enumerate} 
\end{prop}
\begin{proof}
Assume (i). Since $R$ is left semi-hereditary if and only if  $_RR$ is $\aleph_0$-projective, by  Albrecht's result quoted before Fact \ref{Alb}, every projective left $R$-module decomposes into a direct sum of finitely generated ideals. By (i), these are projective---even $\aleph_0$-projective. An old argument (going back at least  to Baer 1937) shows that any such direct sum $P=\bigoplus_i P_i$ is  $\aleph_1$-projective. Namely, if $c_0, c_1, \ldots$ are countably many elements in $P$, choose a direct summand $P_{i_0}$ of $P$ containing $c_0$, write $P=P_{i_0}\oplus Q_1$ and $c_1=c_1' + q_1$ accordingly and choose a direct summand $P_{i_1}$ of $Q_1$ containing $q_1$ and so on. Eventually, all $c_i$ will be contained in the direct summand $P'=\bigoplus_{k<\omega} P_{i_k}$ of $P$, all of whose summands $P_{i_k}$ are finitely generated  and  $\aleph_0$-projective. Then the submodules generated by any of the elements $c_0$, $c_i'$, or $q_i$ are also projective, hence so is their direct sum, which is the submodule generated by the original $c_i$. 
This proves (ii) (and the converse is trivial). 

The direction from left to right in (iii) is  easily derived from (ii). The other direction is always true by Corollary \ref{aleph1}.

 (iii) $\Leftrightarrow$ (iv) is trivial,  while  (iv) $\Rightarrow$ (v) follows from the fact that being $\aleph_1$-projective is obviously hereditary.

Assuming (v), notice that every countably generated left ideal must be  flat and Mittag-Leffler, hence (flat and pure-) projective. So (i) follows.
 \end{proof}

The next result is known for  hereditary rings  \cite[Thm.7.58]{L}, but its proof works equally well for  countably hereditary rings.

\begin{cor}\label{noeth}
 A left countably hereditary ring of finite uniform (=Goldie) dimension is left noetherian. In particular, a left  countably hereditary left Ore domain is left noetherian. Thus, a left countably hereditary left B\'ezout domain is a left principal ideal domain.
\end{cor}
\begin{proof}
 In the proof of   \cite[Thm.7.58, p.267]{L}, it suffices to show that every \emph{countably} generated  left ideal $I$ is projective and then, by \cite[Prop.7.60]{L}, finitely generated (because of finite uniform dimension). 
 
 For the last statement  note that a left  B\'ezout domain is left Ore  \cite[Prop.II.1.8]{S} and, if also left noetherian, a left PID. 
\end{proof}

\begin{fact}\label{*} 
\begin{enumerate}[\upshape(1)]
  \item \mbox{}\cite[Prop.6.2+Cor.6.3]{habil}. The above conditions {\rm (A) -- (D)} are equivalent  for every left module  over a left or right semi-hereditary ring (and imply  {\rm (F)} over a left semi-hereditary ring).
\item  \mbox{}\cite[Cor.6.5  (for countable rings)]{habil}. The above conditions {\rm (A) -- (E)} are equivalent (and imply  {\rm (F)}) for every left module  over a left countably hereditary ring.
\end{enumerate}
\end{fact}
 \begin{proof} (1) (D) $\Rightarrow$ (A)  $\Leftrightarrow$ (B)  $\Leftrightarrow$ (C)  have been mentioned already. For (C) $\Rightarrow$ (D), note that every finite subset is contained in a (countably generated) projective pure submodule, which, being over a left or right semi-hereditary ring, decomposes into a direct sum of finitely generated projectives, Fact \ref{Alb}. So, every finite subset of $M$ is contained in a finitely generated projective submodule which is a direct summand of a pure submodule, so itself pure in $M$. Clearly, (F) is implied by (D) over a left semi-hereditary ring.
 
 (2) Obviously, (E) is implied by (C) (and (iii) of the proposition). Invoke Corollary \ref{aleph1} for the converse.
\end{proof}

Part (3) below, at least  for countable PIDs,  was contained in  \cite[after Cor.6.5]{habil}.

\begin{cor}\label{Bezout}\mbox{}
Denote by {\rm (A$^\prime$) -- (F$^\prime$)} the conditions from Conditions \ref{cond} with `projective' replaced by `free' (wherever possible).

\begin{enumerate}[\upshape(1)]
  \item   Conditions {\rm (A$^\prime$) -- (D$^\prime$)} are equivalent and imply  {\rm (F$^\prime$)} for every left module  over a semi-fir.
\item   Conditions {\rm (A$^\prime$) -- (E$^\prime$)} are equivalent (and imply  {\rm (F$^\prime$)}) for every left module  over a left $\aleph_0$-fir.
\item A left module over a left countably hereditary left or right B\'ezout domain (in particular, over a left or right principal ideal domain) is flat Mittag-Leffler if and only if it is $\aleph_1$-free.
\end{enumerate}

\end{cor}
\begin{proof}
By Fact \ref{Alb}, projective=free over semi-firs, so (1) and (2) follow from   Fact \ref{*} (use Remark \ref{aleph0fir} for (2)). (3) is a special case of  (2), since one-sided B\'ezout domains are semi-firs, and  
 left PIDs are left hereditary and  B\'ezout.
\end{proof}
 
 (Remember, by Corollary \ref{noeth}, a left countably hereditary left B\'ezout domain is already a left PID.)
 
 The special case for the ring $\mathbb{Z}$, i.e., for abelian groups, was  derived from Pontryagin's Criterion in \cite[Prop.7]{AF}.
 
\begin{cor}[Azumaya--Facchini]\mbox{}
 
 An abelian group is torsion-free and Mittag-Leffler if and only if it is $\aleph_1$-free.
\end{cor}

\begin{rem} It is not hard to see directly from Fact \ref{MThm} and the definition of $\flat$-atomicity that 
any such group is, in Baer's terminology, a homogeneous  torsion-free group of null type, i.e., all characteristics are equivalent to the characteristic $(0,0,0, \ldots)$, cf.\ \cite[Ch.XIII]{F}. (This is also known as   type $\mathbb{Z}$.)
 \end{rem}

\subsection{Von Neumann regular rings}

Regular rings are characterized by the fact that all modules are flat and also by the fact that all modules are absolutely pure (on either side). 

\begin{rem}\label{c-hered-reg}
 Since pure submodules of Mittag-Leffler modules are Mittag-Leffler, Proposition \ref{c-hered} gives another proof of the known fact  that  von Neumann regular rings are countably  hereditary, cf.\ \cite[Exples.2.32(e)]{L}. In particular, countable regular rings are hereditary \cite[Exples.2.32(e)]{L}, and projectives over regular rings are $\aleph_1$-projective  \cite[Cor.6.8]{habil}. 
\end{rem}

\begin{fact} \label{reg} \mbox{}\cite[Thm.6.7]{habil}.
 Over a regular ring, the following concepts are all the same: Mittag-Leffler, flat Mittag-Leffler, $\sharp$-Mittag-Leffler, $\flat$-Mittag-Leffler, $\aleph_1$-$*$-projective,  $\aleph_0$-$*$-projective, $\aleph_1$-projective,  $\aleph_0$-projective.
\end{fact}
\begin{proof}
 As regular rings are countably hereditary, all but the last property is taken care of by Fact \ref{*}(2), and we have to verify only that it implies any of the others.
 But every embedding being pure, 
  the unstarred properties entail the starred ones.
\end{proof}

\section{Torsion-free and divisible modules}
Recall that a submodule $M$  is \emph{RD-pure} in  a left $R$-module $N$ if $M$ is \emph{relatively divisible} in $N$ in the sense that $rN\cap M=rM$ for every $r\in R$.

Torsion-free and divisible modules (as defined in the introduction)  play the same role with respect to RD-pure-exact sequences as flat and absolutely pure modules play with respect to pure-exact sequences. Namely, a module is torsion-free (resp., divisible) if and only if every short exact sequence ending (resp., beginning) in it is RD-pure \cite[Prop.3]{Hat}.

What does elementary  duality have to say about this? The first answer to this question is on the level of individual modules and their character duals, derived as a special case of Fact \ref{PRZ}(1) (which can also be seen directly from the Ext and Tor descriptions, as in \cite[Prop.1.4]{DF}).

\begin{prop}\label{newL}
\begin{enumerate}[\upshape(1)]
 \item A module $M$ is  torsion-free  if and only if  its character module $M^*=\Hom_\mathbb{Z}(M, \mathbb{Q}/\mathbb{Z})$ is  divisible.
 \item A module  $M$ is divisible whenever  its character module  $M^*$  is torsion-free.
\end{enumerate}
\end{prop}
\begin{proof}
 Use Fact \ref{PRZ}(1)  and the duality of the F- and A-sentences axiomatizing the two concepts as given in clauses (iii) of Facts \ref{flat} and \ref{abspure}.
\end{proof}

In analogy with W\"urfel's theorem as stated after Fact \ref{LambWurf}, we will see that the converse of (2) is not true in general. Moreover, we are going to exhibit for which rings exactly it holds.

\begin{cor}\label{subdual}
 The definable subcategory generated by $\Tf$ is contained in the dual of the definable subcategory generated by $\Div$.
\end{cor}
\begin{proof}
As  $\D\psi\leq_{M}\D\phi$ if and only if $\phi\leq_{M^*}\psi$ (see Section \ref{elemdual}), the first part of the proposition shows that $\phi\leq_{\Div}\psi$  implies  $\D\psi\leq_{\Tf}\D\phi$. Hence  $\D\psi\leq_{\D\langle\Div\rangle}\D\phi$   implies  $\D\psi\leq_{\Tf}\D\phi$. Fact \ref{inclusion} now yields $\langle\Tf\rangle\subseteq \D\langle\Div\rangle$.
\end{proof}

\begin{cor}
 Every $\Div$-Mittag-Leffler module is  $\Tf$-atomic.
\end{cor}
\begin{proof}
As in the previous proof, $\leq_{\D\langle\Div\rangle}$ implies $\leq_{\Tf}$. Thus $\D\langle\Div\rangle$-atomic, which is the same as $\Div$-Mittag-Leffler, implies $\Tf$-atomic. \end{proof}

The converse will follow over right P-coherent rings from the next result showing that over such rings $\Div$ and $\Tf$ form dual definable subcategories.

The proof is a version of the proof of W\"urfel's theorem  given in \cite[Thm.4.4]{PRZ2}. Considering only the A-statements axiomatizing divisibility (as opposed to full absolute purity) allows us to derive a result that corresponds exactly to what is known for absolutely pure right  modules and flat left modules and right coherent rings, Fact \ref{ES}. Precursors were, first of all,   \cite[Prop.8(ii)]{Hat}, which proved (i) $\Leftrightarrow$ (ii), then   \cite[Prop.1.5]{DF}, where (iv) was  verified for right coherent rings, and finally \cite[Thm.2.7]{MD}, which exhibited the
 equivalence of (i)--(iv) (together with the equivalent condition of existence of $\Tf$-preenvelopes for every left $R$-module).

\begin{thm}\label{divtfD}
\begin{enumerate}  [\upshape (1)]
 \item The following are equivalent for any ring $R$.
 
 \begin{enumerate}  [\upshape (i)]
\item $R$ is right P-coherent.
\item The class $\Tf$ of torsion-free left $R$-modules is closed under direct product.
\item The class $\Div$ of divisible right $R$-modules is closed under direct limit.
\item A right $R$-module $M$ is divisible if and only if its character module $M^*$ is a torsion-free left $R$-module.

\item $\Tf$  is (axiomatizable by finitary pp implications and hence) a definable subcategory.
\item  $\Div$ is (axiomatizable by finitary pp implications  and hence) a definable subcategory.
\end{enumerate} 
\item In this case,  the definable subcategories  $\Div$  and  $\Tf$ are (elementarily) dual.
\end{enumerate} 
\end{thm}
\begin{proof}
First of all, (ii) $\Leftrightarrow$ (v) and  (iii) $\Leftrightarrow$ (vi), because (ii) and, resp., (iii) express the only property missing in order to be a definable subcategory (i.e., a subcategory closed under pure submodule (which is true for both), direct limit and direct product). 

(ii) $\Rightarrow$ (i). 
Let
$\rann(r)=\{r_i:i\in |R|\}$ and consider $\rho=(r_i:i\in R)\in R^{|R|}$\,. Then $r\rho=0$, hence, by
torsion-freeness, there must be $s_j\in\rann(r)$ and $\rho_j=(\rho_{ji}:i\in
R)\in{R^{|R|}}$ ($j<k$) such that $\rho=\sum_{j<k}s_j\rho_j$\,. Then
$r_i=\sum_{j<k}s_j\rho_{ji}\in\sum_{j<k}s_jR$ for all $i$\,, whereby
$\rann(r)=\sum_{j<k}s_jR$\,.

For (iii) $\Rightarrow$ (i) one may adjust the proof of the corresponding result about $\sharp$ being definable using reduced products as given in the proof of \cite[Thm.3.4.24]{P2}.

Assuming (i) we show that the F-sentence expressing that $M^*$ is torsion-free becomes finitary, namely an implication of pp formulas. Its truth in  $M^*$ is therefore equivalent to the truth of its dual in $M$, which is the corresponding A-sentence that also becomes finitary for the same reason.

Here are the details. Let $M$ be a divisible  right $R$-module. In order to show that $M^*$ is torsion-free, let $rf=0$ with $f$ a character on $M$ and $r$ in $R$ (the other direction of (iv) is always true by Proposition \ref{newL}(2)). Write the right annihilator of $r$ as $I=\sum_{i<n}r_iR$. We have to show that $I$ divides $f$. This is equivalent to saying that $f$ satisfies the formula $\exists y_0 \dots y_{n-1} (x=\sum_{i<n} r_i y_i)$. It suffices therefore to show that  $M^*$ satisfies the pp implication $rx=0\to \exists y_0 \dots y_{n-1} (x=\sum_{i<n}r_i y_i)$ (which is what the original F-sentence is equivalent to under the assumption on $I$). Its dual is $(\bigwedge_{i<n} x r_i=0) \to r|x$, which is satisfied in $M$ because of  divisibility (for the antecedent is equivalent to $xI=0$), as desired. This proves that (i) implies (iv), (v), and (vi) (and thus the equivalence of (i) -- (iii) and (v) and (vi)).

For the remaining direction (iv) $\Rightarrow$ (i), assume $I :=\rann(r)$ is  not finitely generated. We exhibit a divisible module whose character dual is not torsion-free.

Let $I_i$ be a list of all finitely generated right ideals inside $I$. Consider $M_i= R/I_i$ for all $i$ and let $A_i$ be the annihilator of $I$ in $M_i$. As $1+I_i\not\in A_i$, there's a character $f_i$ on $M_i$ annihilating all of $A_i$ but not  $1+I_i$. 

Now form
the direct sum $M$ of all $M_i$. Since the annihilator $A$ of $I$ in $M$ is  the direct sum of all the $A_i$, the character $f$ on $M$ which is the sum of all the $f_i$ annihilates all of $A$. 

Choose any divisible module $N$ extending $M$ (e.g., its injective hull) and let $B$ be  the annihilator of $I$ in $N$. As $B\cap M=A$, setting $f'=0_B+f$ defines  a map $f'$ on $B+M$ that extends $f$ and annihilates $B$. Using the injectivity of the abelian group $\mathbb{Q}/\mathbb{Z}$ we finally extend $f'$ to a character $g$ on $N$. Then $g$ annihilates $B$, but none of the $1+I_i\in M_i\subseteq M\subseteq N$.

Since  $B=\ann_N I$ and $I$ is the right annihilator of $r$ (i.e., the subgroup defined in $_RR$ by $rx=0$), the description of pp subgroups of absolutely pure modules, Fact \ref{abspure}, yields that $B$ is the subgroup defined in $N$ by the dual $r|x$ of $rx=0$. In other terms, $B=Nr$. Thus $g(Nr)=0$, which means that $rg=0$, i.e., $g$ satisfies  $rx=0$ in $N^*$.

If $N^*$ were torsion-free, there would be finitely many $s_k$ in $I$ and $g_k$ in $N^*$ such that $g=\sum_k s_k g_k$. Let $I_i$ be the right ideal generated by these $s_k$. Then clearly $g(1+I_i)=0$, contradiciting the choice of $g$.

(2) In view of Fact \ref{inclusion}, it suffices to verify that $\phi\leq_{\Div}\psi$ if and only if $\D\psi\leq_{\Tf}\D\phi$. 

The direction from left to write follows from Proposition \ref{newL}(1): to verify $\D\psi\leq_{\Tf}\D\phi$, let $M\in\Tf$. Since then $M^*\in\Div$, we have  $\phi\leq_{M^*}\psi$, hence $\D\psi\leq_{M}\D\phi$ by what was said at the end of Section \ref{elemdual} about duality and the character dual.

Conversely, if $M\in\Div$, then $M^*\in\Tf$ by (iv), hence, by hypothesis, $\D\psi\leq_{M^*}\D\phi$ and therefore $\phi\leq_{M}\psi$, as desired.
\end{proof}

\begin{cor}\label{DivML}
 Over a right P-coherent ring, a left module is $\Div$-Mittag-Leffler if and only if it is $\Tf$-atomic (and a right module is $\Tf$-Mittag-Leffler if and only if it is $\Div$-atomic).
 \end{cor}

As we know from Fact \ref{MThm}, for the conclusion to be true, all we need is that $\Div$ and $\Tf$ \emph{generate} mutually dual subcategories. This may well be true over a bigger class of rings, or, as in  the case of absolutely pure and flat modules, over all rings, see Fact \ref{IvoH}---however, we don't know. (One point is that the proof for absolutely pure and flat modules relies on the description in them of \emph{all} pp subgroups, which we don't seem to have here.)

\begin{ques}\label{ques}
What are the rings over which the definable subcategories generated by $\Div$ and $\Tf$ are mutually dual?
\end{ques}

\section{Torsion theory}

We are interested in the torsion theory cogenerated by $\Tf$ as considered in \cite{Hat} (and follow the terminology of 
 \cite[Ch.VI]{S}). In particular, we call a module $T$  \emph{torsion} if $\Hom(T, F)=0$ for every $F\in\Tf$ \cite[p.153]{Hat}. Denote the class of (left) torsion modules by $\To$.

Hattori proves that $\Tf$ is a torsion-free class if and only if  the ring is right PP. More precisely, \cite[Prop.5]{Hat} verifies that  $\Tf$ is always closed under extension,  \cite[Prop.7]{Hat} that it is closed under submodule if and only if the ring is right PF, and   \cite[Prop.8(ii)]{Hat} that it is closed under direct product if and only if it is right P-coherent. Now recall from Section \ref{ring}  that PF + P-coherent = PP. See also  \cite[Prop.13]{Hat}.

So over a right PP ring  $\Tf$ cogenerates the  torsion theory $(\To, \Tf)$. But just as in Hattori's work, several of our results on $(\To, \Tf)$ do not depend on it actually being a torsion theory. 

For example, it is shown in \cite[Prop.14]{Hat} that the tensor product of a divisible module with a torsion module is zero. As products of divisible modules are divisible \cite[Prop.8$^\prime$]{Hat}, this trivially shows 
 
\begin{prop}\label{torsionisML}
Any torsion module is $\Div$-Mittag-Leffler (=$\Tf$-atomic), over any ring. 
\end{prop}

The radical associated with  the  torsion theory $(\To, \Tf)$ is denoted $\T$, where $\T(M)$ is the largest torsion submodule of $M$. This  exists, by \cite[Cor.\ to Prop.12]{Hat}, even for arbitrary rings (where it may not be a radical), which means that  $\T$ is always a preradical. 

We now turn to the easy direction of our final result.

\begin{cor}\label{EasyDirect}
 Suppose $M$ is a module with $\T(M)$ pure in $M$.
 
 If  $M/\T(M)$ is $\Div$-Mittag-Leffler, so is $M$.
\end{cor}
\begin{proof} 
By  Proposition \ref{torsionisML}, the torsion part $\T(M)$ is $\Div$-Mittag-Leffler. So, if also  $M/\T(M)$ is  $\Div$-Mittag-Leffler, then, being a pure extension of these, so is $M$ (cf.\ Fact \ref{MThm}(3)). 
\end{proof}

The remainder of the paper is  devoted to  the converse of this, see especially Theorem \ref{thmRD}.

Consider two subfunctors of $\T$. First of all, there is the classical $\tor$,  which collects all the  $\Sreg$-torsion elements (with $\Sreg$, the multiplicatively closed subset of all regular elements, i.e., non-zero-divisors, of $R$). In other terms, $\tor(M)=\{a\in M\, |\,  \textrm{$sa=0$ for some $s\in \Sreg$}\}$. This is a preradical if  $\Sreg$ is a left denominator set, i.e., if the ring is left Ore \cite[Prop.II.1.6]{S}. It is then actually a torsion radical (i.e., also hereditary)  \cite[Exple.VI.1.2, p.138]{S}.

 And then we have an `infinitary pp functor,' $\s$, which, in a given module $M$,  by definition singles out the sum of all pp subgroups $\theta(M)$ where $\theta$ runs over all (1-place) pp formulas that are $\Tf$-equivalent to $x=0$. Call $\s$ the \emph{elementary torsion preradical} of $M$.
That it is functorial follows from the fact that pp formulas are preserved by homomorphisms. It is easy to see that it is in fact a radical when $\s(M)$ is a pure submodule of $M$ for all $M$. 

\begin{lem} For any ring, $\s$ is a subpreradical of the preradical $\T$.

 For every module $M$, we have $\tor(M)\subseteq \s(M)\subseteq\T(M)$.
\end{lem}
\begin{proof} 
Let $\Phi$ be the set of 1-place pp formulas that are $\Tf$-equivalent to $x=0$. This set is closed under addition of pp formulas, so also $\s(M)$ is closed under addition. To see that it is closed under scalar multiplication, let   $r\in R$. For every pp formula $\theta\in\Phi$, clearly also the pp formula that defines $r\theta(M)$ (in \emph{every} module $M$), i.e., the formula $\exists y(x=ry\wedge \theta(y))$, is in $\Phi$. Thus $\s(M)$ is closed under $r$, and, as $r$ was arbitrary, a submodule of $M$. That $\s$ is a preradical now follows from the fact that pp formulas are preserved under homomorphisms.

To see that $\s(M)\subseteq\T(M)$, simply  note that $\s$ is zero in torsion free modules (by definition), so $\s(M)$ is a torsion submodule of $M$ and must therefore be contained in the largest one, $\T(M)$.

It remains to verify that $\tor(M)\subseteq \s(M)$. But, as is easily seen, a torsion-free module (in our sense) has no $\Sreg$-torsion elements, whence $\tor(M)$ is a subsum of $\s(M)$. \end{proof}

\begin{prop}\label{s=T} $\s(A)=\T(A)$, for every countably generated $\Tf$-atomic module $A$.
\end{prop}
\begin{proof} To prove the inclusion from right to left, let the pp type of $a\in \T(A)$ be $\Tf$-generated by $\theta$. We claim $\theta\sim_{\Tf} (x=0)$ (and hence $a\in \theta(A)\subseteq\s(A)$, as desired).

To this end, let $F$ be any torsion-free module and $c\in \theta(F)$. We need to show $c=0$. By Proposition \ref{freereal}, there's a map $(A, a)\to (F, c)$. Its restriction to the torsion module $\T(A)$ must be $0$, hence $c=f(a)=0$, as claimed.
\end{proof} 

\begin{rem}
 It is obvious that, mutatis mutandis, the same holds true for any torsion theory $(\cal T, \cal F)$  with radical $\T$.
\emph{
Define the elementary torsion preradical via $\s_{\cal F}(M)=\sum_{\sigma\sim_{\cal F} (x=0)} \sigma(M)$. Then $\s_{\cal F}$  is a subpreradical of the radical $\T$ such that $\s_{\cal F}(M)=\T(M)$ for every countably generated $\cal F$-atomic module $M$.
}
\end{rem}

In order to extend this to $\Tf$-atomic modules of arbitrary size, we assume the torsion theory to be hereditary. 

\begin{cor}\label{s=T'} 
 If  the torsion theory  $(\To, \Tf)$ is hereditary (i.e., $\To$ is), then $\s(M)=\T(M)$, for every  $\Tf$-atomic $R$-module $M$.
\end{cor}
\begin{proof} Let $a\in \T(M)$. Using Proposition \ref{K-RG}, choose a countably generated   $\Tf$-atomic submodule $A\subseteq M$ containing $a$.  By heredity, $\T(A)= \T(M)\cap A$ \cite[Prop.1.7 and 3.1]{S}. So $a\in \T(A)$, hence $a\in\s(A)\subseteq \s(M)$ by the proposition.
\end{proof} 

\begin{rem}\label{hered}
It is well known that a torsion theory $(\cal T, \cal F)$ is hereditary if and only if it can be cogenerated by an injective torsion-free module (see \cite[Prop.15]{Hat} or \cite[Prop.VI.3.7]{S})  if and only if $\cal F$ is closed under injective envelopes \cite[Prop.VI.3.2]{S}---or, in plain terms,  every submodule of any torsion module is torsion.
\end{rem}

The next lemma is the key fact and the reason for the introduction of the elementary torsion preradical.

\begin{lem}\label{tf-factor}
 Let $M$ be a $\Tf$-atomic module with $\s(M)=\T(M)$ and  such that this submodule is pure in $M$.  Then $M/\T(M)$ is  $\Tf$-atomic as well.
\end{lem}
\begin{proof}
 Let $\bar{b}+\T(M)$ an arbitrary tuple in  $M/\T(M)$ and $\theta$ a pp formula that $\Tf$-generates the pp type of $\bar{b}$ in $M$. We claim that this formula  $\Tf$-generates also the pp type of $\bar{b}+\T(M)$  in  $M/\T(M)$ (i.e., independently of the choice of the representative $\bar{b}$).
 So let  $\phi$ be a pp formula that $\bar{b}+\T(M)$ satisfies. It remains to  verify $\theta\leq_{\Tf}\phi$. 
 
 Using purity choose a preimage $\bar{b}'$ in $\phi(M)$, i.e., such that $\bar{b}-\bar{b}'$ is in $\T(M)$. More precisely, if $\bar{b}$ has entries $b_i$ ($i<n$), then $\bar{b}'\in\phi(M)$ has entries $b_i'$  such that $b_i-b_i'\in\T(M)$ for $i<n$.  By hypothesis, each $b_i-b_i'$ satisfies some pp formula $\sigma_i\sim_{\Tf} (x=0)$, hence $\bar{b}-\bar{b}'$ satisfies an $n$-place pp formula $\sigma\sim_{\Tf} (\bar{x}=\bar{0})$ in $M$ (namely, their conjunction). Then $\bar{b} \in \bar{b}' +\sigma(M)\subseteq\phi(M)+\sigma(M)$, hence $\phi+\sigma$ is in the type of  $\bar{b}$ in $M$. Hence $\theta\leq_{\Tf}\phi+\sigma$ and thus  $\theta\leq_{\Tf}\phi$, for  $\sigma\sim_{\Tf} (\bar{x}=\bar{0})$.
\end{proof}

\section{RD-rings}
 An \emph{RD-ring} is a ring over which RD-purity \emph{is} purity. This is, as Menal and V\'amos noticed, a two-sided notion. The class of all RD-rings is not a small one. It contains all commutative Pr\"ufer domains. Moreover, Warfield showed that all commutative rings whose localizations at maximal ideals are  valuation rings (call these \emph{commutative Pr\"ufer rings}) are RD, and he proved the converse for commutative rings. Further,   all von Neumann regular rings are RD, and so are all serial rings and all Dedekind prime rings, in particular, the first Weyl algebra $A_1$ over a field of characteristic $0$ is an RD-domain. See \cite{PPR} or  \cite[Sect.2.4.2]{P2} for this and more on RD-rings.

As mentioned before,  Hattori noticed, and this is easy to prove directly from the definitions, that the torsion-free modules are, so to speak, the RD-flat modules and the divisible modules are the absolutely RD-pure ones. Hence, over an RD-ring, the classes $\Tf$ and $\flat$ coincide, as well as do $\Div$ and $\sharp$, \cite[Prop.3]{Hat} (on whichever side we want, as RD is  two-sided). Thus, our initial duality, over  P-coherent rings, of $\Div$ and $\Tf$  now, over RD-rings, becomes the usual duality of  $\sharp$ and $\flat$, for which we no longer need P-coherence, cf.\ Fact \ref{IvoH}. 

Recall that absolutely pure right modules over a right noetherian ring are injective. Hence, over a right noetherian RD-ring,  divisible right modules are injective.

\begin{rem}
 In view of  Proposition \ref{s=T},  all results in this section hold for \emph{countably generated} $M$ even when the torsion theory is not hereditary.
\end{rem}

\subsection{Semi-hereditary RD-rings}
Even though we no longer need P-coherence for the duality of  $\sharp$ and $\flat$, in order to have $\T(M)$ always pure in $M$, which seems to be an essential part of our proof, we are led to assume  the ring to be   right PP so that Hattori's torsion theory is available (for left modules), for then $M/T(M)$ is torsion-free and therefore  $\T(M)$ must  be RD-pure (hence pure) in $M$. Note that a right PP-ring which is also RD is automatically right semi-hereditary \cite[Prop.2.21]{PPR}. So we work for the remainder over a right semi-hereditary RD-ring. We further assume the torsion theory to be hereditary, so that   Corollary \ref{s=T'}  applies.
Invoking Lemma \ref{tf-factor}  we now obtain
 
\begin{lem}
 Let $R$ be a right semi-hereditary RD-ring with $(\To, \Tf)$  hereditary.
 
Then $M/\T(M)$ is 
  $\Tf$-atomic for \emph{every}  $\Tf$-atomic module $M$.\qed
 \end{lem}
 
\begin{thm}\label{thmRD}
Suppose $R$ is an  RD-ring  with $(\To, \Tf)$  hereditary and $M$ is  a left $R$-module with  $F:= M/\T(M)$ its largest torsion-free factor (which is flat).
\begin{enumerate}[\upshape (1)]
\item If $R$ is  right semi-hereditary,
then   $M$ is  $\sharp$-Mittag-Leffler if and only if  $F$ satisfies the equivalent conditions  {\rm (A) -- (D)} (from Conditions \ref{cond}).

\item   If $R$ is  left countably hereditary and right  semi-hereditary,  
then   $M$ is  $\sharp$-Mittag-Leffler if and only if  $F$ satisfies the equivalent conditions  {\rm (A) -- (E)}.
\end{enumerate}

\end{thm}
\begin{proof} First of all, by the introductory discussion, the factor module $F$ is, under the hypotheses, indeed flat. 
 
 Since $\sharp=\Div$, the direction from right to left follows from  Corollary \ref{EasyDirect}  (over any ring and for any  torsion theory---so long as $\T(M)$ is pure in $M$).\footnote{Nothing is gained from this generality in the commutative domain case, see the concluding remarks.}
For the converse apply Fact \ref{*}.
\end{proof}

\begin{rem} As von Neumann regular rings are semi-hereditary and RD, we see that all modules are torsion-free and divisible. So `$\Div$-Mittag-Leffler' and `$\Tf$-Mittag-Leffler' may be added to the equivalent concepts in  Fact \ref{reg} (and the rest of the theory is redundant, since always $\T(M)=0$).
\end{rem}

\begin{cor}
 Suppose $R$ is a Dedekind prime ring  with $(\To, \Tf)$  hereditary and $M$ is a left  $R$-module with  $F:= M/\T(M)$ its largest torsion-free factor (which is flat).

Then  $M$ is  $\sharp$-Mittag-Leffler if and only if   $F$ satisfies the equivalent conditions {\rm (A) -- (E)} from Conditions \ref{cond}.
\end{cor}
\begin{proof}
  A Dedekind prime ring is RD \cite[Cor.2.11]{PPR} (whence $F$ is flat) and, by definition, two-sided hereditary. (Is the heredity hypothesis on the torsion theory also two-sided?---it is for domains, see Lemma \ref{T=t} below.)
\end{proof}

\subsection{RD-domains}
Examples are, as mentioned,  all commutative Pr\"ufer domains (and these are the only commutative RD-domains), but there are  non-commutative examples, see Corollary \ref{DedDom} and Remark \ref{3}(1) below and \cite[Sect.5]{PPR}. Since  $\mathbb{Z}$ is an RD-domain, everything applies to abelian groups.

\begin{rem}\label{thmRDdom} RD-domains are (left and right) semi-hereditary and coherent \cite[Cor.5.1]{PPR}, hence
Theorem \ref{thmRD}(1) applies to any RD-domain with   hereditary torsion theory $(\To, \Tf)$. Part (2) applies if  $R$ is, in addition, left countably hereditary.
\end{rem}

\subsection{Ore domains and hereditary torsion theories}
In view of the preceding remark, we now consider RD-domains over which the torsion theory  $(\To, \Tf)$ is hereditary.

  \cite[Prop.18]{Hat} noticed that an arbitrary domain  is left Ore  if and only $\T$ is the usual $\Sreg$-torsion radical $\tor$ (for left modules).
  We can say slightly more for RD-domains.
 
\begin{lem}\label{T=t} The following are equivalent for any RD-domain $R$.
 \begin{enumerate}  [\upshape (i)]
 \item  $\T=\tor$.
 \item  $R$ is left Ore.
 \item  $(\To, \Tf)$ is hereditary.
 \item The same conditions on the right.
 \end{enumerate}
\end{lem}
\begin{proof}
(i) $\Rightarrow$ (iii) is obvious (and well-known).  

(iii) $\Rightarrow$ (ii).
By \cite[Prop.VI.3.2]{S}, we have (iii) exactly when $\Tf$ is closed under injective envelopes. In that case, there is a non-zero torsion-free and injective left $R$-module, which, under the hypothesis of RD, implies the  left  and right Ore properties by \cite[Lemma 5.2+Prop.5.3]{PPR}. 

(ii) $\Rightarrow$ (i). When (ii) holds (on the left, say), then, for any left $R$-module $M$, the factor  $M/\tor(M)$ is torsion-free in the usual sense \cite[ex.10.19]{L}, hence also in our sense, as those two are the same for domains. Hence $\T(M/\tor(M))=0$. Since $\tor(M)\subseteq\T(M)$, these must be the same.
\end{proof}

Call such rings \emph{RD-Ore domains}. 

\begin{thm}\label{thmRDOre} Let $R$ be an  RD-Ore domain.
 \begin{enumerate}[\upshape (1)]
\item The following are equivalent for any $R$-module $M$.
         \begin{enumerate}  [\upshape (i)]
	\item $M$ is $\sharp$-Mittag-Leffler (= $\Div$-Mittag-Leffler).
	\item $M/\tor(M)$ is $\sharp$-Mittag-Leffler.
	\item $M/\tor(M)$ is Mittag-Leffler.
	\item  $M/\tor(M)$ is $\aleph_1$-$*$-projective.
 	\item $M/\tor(M)$ is $\aleph_0$-$*$-projective.
	\end{enumerate}
 \item   If $R$ is, in addition, left countably hereditary (equivalently, left noetherian),  another equivalent condition is:

\hspace{-1em}{\rm (vi)}  $M/\tor(M)$ is $\aleph_1$-projective.
\end{enumerate}  
\end{thm}
\begin{proof}
In view of the  above discussion and Lemma \ref{T=t},  both parts  follow from Theorem \ref{thmRD}.  (By Corollary \ref{noeth}, an RD-Ore domain is left countably hereditary if and only if it is left noetherian.)
\end{proof}

\begin{cor}\label{DedDom}
If $M$ is a module over  a Dedekind domain (i.e.,   two-sided  hereditary  and noetherian domain without idempotent   two-sided ideals), then {\rm (i) -- (vi)} above are equivalent.
\end{cor}
\begin{proof}
 Any Dedekind prime ring is RD \cite[Cor.2.11]{PPR} and Ore \cite[Prop.II.1.7]{S}.
\end{proof}

\begin{rem}\label{3}
\begin{enumerate}[\upshape (1)]
 \item The corollary applies, in particular, to $\sharp$-Mittag-Leffler (= $\Div$-Mittag-Leffler) modules over the first Weyl algebra $A_1$ over a field of characteristic $0$, since it is a Dedekind domain. See \cite[5.2.11]{MR} for this and other exmples, cf.\ also \cite[Rem.2.13]{PPR} or \cite{P2}.
 \item For torsion-free modules over hereditary RD-Ore domains, the equivalence of (iii) -- (vi) above was proved in \cite[Thm.6.13]{habil}, as well as two more equivalences on $M/\tor(M)$:
 
 (vii) Every finite subset   is contained in a finitely presented pure submodule.
 
 (viii) Every finite rank submodule is finitely generated projective (hence finitely presented).
\end{enumerate}
\end{rem}

 Next we impose the B\'ezout condition onto our RD-domain (which makes it an Ore domain again, see e.g.\ \cite[Prop.II.1.8]{S}). For RD-domains, this is also a two-sided property, see the discussion below.
 
Note that we would gain nothing from relaxing B\'ezout to semi-fir, for even $2$-firs satisfying the Ore condition  are   B\'ezout already, see the discussion before Fact \ref{Alb}. (For RD-domains that are not semi-firs and semi-firs that are not RD, see \cite[Rem.5.6]{PPR}.)

\begin{cor}\label{corBez}
\begin{enumerate}[\upshape (1)]
  \item If $R$ is a B\'ezout domain that is also RD, then {\rm (1)} of the theorem holds true with `projective' everywhere replaced by `free.'  

 \item  If $R$ is a left principal ideal domain that is also RD, then a left module $M$ is $\sharp$-Mittag-Leffler (= $\Div$-Mittag-Leffler) if and only if  $M/\tor(M)$ satisfies conditions {\rm (ii) -- (vi)} with  `projective' everywhere replaced by `free.'

  \end{enumerate}
 \end{cor}
\begin{proof}
 Apply Corollary \ref{Bezout}.
\end{proof}

\begin{cor}\label{AG}
 An abelian group $M$ is $\sharp$-Mittag-Leffler  if and only if $M/\tor(M)$ is $\aleph_1$-free.
\end{cor}

\begin{rem}
 While all left modules over a left noetherian ring are $\flat$-Mittag-Leffler \cite{Goo}, we conclude that this is not true for the dual notion: there are abelian groups that are not $\sharp$-Mittag-Leffler, take for instance $\mathbb{Q}$. This follows from Theorem \ref{alreadyML} and the known fact that $\mathbb{Q}$ is not Mittag-Leffler. One can verify this directly however: the same argument that shows  $\mathbb{Q}$ is not Mittag-Leffler also shows that it is not $\sharp$-Mittag-Leffler.
\end{rem}

By Corollary \ref{noeth}, a left countably hereditary RD-B\'ezout domain is  a left PID, whence the last two cases in Corollary \ref{Bezout} collapse into one once the RD property is imposed---namely, into part (2) of Corollary \ref{corBez}. Nevertheless, the hypotheses in (2) are not redundant, since there are left PID's that are not right Ore \cite[p.53]{S}; such rings  cannot be RD, for the Ore property is two-sided in RD-rings  (and  left PID's  are certainly left Ore---even left B\'ezout domains are \cite[Prop.II.1.8]{S}). 
While an RD-domain that is  B\'ezout on one side is so on the other \cite[Cor.5.5]{PPR}, it is not known whether two-sided B\'ezout domains have to be RD  \cite[Quest.6]{PPR}.  However, a left PID that is right B\'ezout is indeed RD (and Ore) \cite[Prop.5.7]{PPR}.

It is unknown whether there are any RD-domains that are not Ore  \cite[Quest.5]{PPR}. So it is unclear if $T$ and $t$ can ever  differ in an RD-domain and hence, whether Theorems \ref{thmRD} and  \ref{thmRDOre} say anything different about RD-domains. 

Cohn and Schofield found  examples of left B\'ezout and right PIDs that are not left PIDs, cf.\ end of \cite{C'} (or of the second edition of \cite{C}). Such rings are right hereditary and right noetherian (hence also Ore) RD-domains that are neither left noetherian nor left hereditary \cite[Rem.5.8]{PPR}.

\section{Concluding remarks}

One may ask what happens over \emph{one}-sided Ore domains (which are certainly not  RD). Our proofs used the purity of 
$\T(M)$, which is guaranteed over  RD-rings. For commutative domains this is  the only way. Namely, then torsion parts are  pure submodules if and only if the domain in question is Pr\"ufer  \cite[Prop.I.8.12]{FS}, hence indeed RD. What can be said about this if the ring contains zero-divisors or is no longer commutative, I do not know.

\end{document}